\documentclass[12pt,reqno]{amsart}
\usepackage{fullpage, times, amssymb, amsthm, amsmath, amsfonts}
\usepackage{graphicx}
\usepackage{times}
\usepackage{enumitem}
\usepackage{mathrsfs}
\usepackage{mathtools}
\usepackage{xfrac}
\usepackage{float}
\usepackage[all,cmtip]{xy}
\usepackage[usenames,dvipsnames]{xcolor}
\usepackage[margin=1.5cm]{geometry}
\usepackage{hyperref}
\hypersetup{colorlinks=true, citecolor=blue, linkcolor=blue, urlcolor=blue, pdfstartview=FitH, pdfauthor=Bober-Goldmakher, pdftitle=Converse to a theorem of Gauss}

\usepackage{tikz}
\usetikzlibrary{matrix,arrows,decorations.pathmorphing}

\begin{document}

\pagestyle{empty}

\parskip0pt
\parindent10pt

\newenvironment{answer}{\color{Blue}}{\color{Black}}
\newenvironment{exercise}{\color{Blue}
\begin{exr}}{\end{exr}\color{Black}}

\newtheorem{theorem}{Theorem}[section]
\newtheorem{prop}[theorem]{Proposition}
\newtheorem{lemma}[theorem]{Lemma}
\newtheorem{cor}[theorem]{Corollary}
\newtheorem{conj}[theorem]{Conjecture}
\newtheorem{defn}[theorem]{Definition}

\newtheorem*{claim}{Claim}
\newtheorem*{MainThmv2}{Theorem \ref{thm:MainThm}, v2.0}

\newtheorem{exr}[theorem]{Exercise}
\newtheorem{example}[theorem]{Example}

\theoremstyle{remark}
\newtheorem*{rmk}{Remark}

\newtheorem*{question}{Question}

\newtheorem*{FundQuest}{The Fundamental Question}
\newtheorem*{FundQuestSeries}{Fundamental Question (for Fourier Series)}
\newtheorem*{FundQuestTransform}{Fundamental Question (for Fourier Transform)}
\newtheorem*{FundQuestGroup}{Fundamental Question (for Fourier Transform on Finite Abelian Groups)}

\makeatletter
\newcommand{\oset}[3][0ex]{%
  \mathrel{\mathop{#3}\limits^{
    \vbox to#1{\kern-2\ex@
    \hbox{$\scriptstyle#2$}\vss}}}}
\makeatother

\renewcommand{\mod}[1]{{\ifmmode\text{\rm\ (mod~$#1$)}\else\discretionary{}{}{\hbox{ }}\rm(mod~$#1$)\fi}}

\newcommand{\act}[2]{{\ifmmode\text{$#1$\ \rotatebox[origin=c]{-90}{$\circlearrowright$}\ $#2$}}}

\DeclareRobustCommand{\leg}
{\genfrac(){0.5pt}{}}

\newcommand{\dtr}{\oset[-.7ex]{\cdot}{-}}
\newcommand{\isom}{\oset
    {\widetilde{\phantom{\longrightarrow}}}
    {\longrightarrow}}

\newcommand*{\dblast}
    {\raisebox{-0.1ex}{*}\llap{\raisebox{-1.25ex}{*}}}

\newcommand{\ns}{\mathrel{\unlhd}}
\newcommand{\wh}[1]{\widehat{#1}}
\newcommand{\whh}[1]{\widehat{\widehat{\vphantom{#1}#1}}}
\newcommand{\wt}[1]{\widetilde{#1}}
\newcommand{\floor}[1]{\left\lfloor#1\right\rfloor}
\newcommand{\abs}[1]{\left|#1\right|}
\newcommand{\ds}{\displaystyle}
\newcommand{\nn}{\nonumber}
\newcommand{\im}{\textup{im }}
\renewcommand{\ker}{\textup{ker }}
\newcommand{\Gal}[1]{\textup{Gal}(#1)}
\renewcommand{\Im}{\textup{Im }}
\renewcommand{\Re}{\textup{Re }}
\renewcommand{\ni}{\noindent}
\renewcommand{\bar}{\overline}
\newcommand{\supp}{\textup{supp}}

\newcommand{\mattwo}[4]{
\begin{pmatrix} #1 & #2 \\ #3 & #4 \end{pmatrix}
}

\newcommand{\vtwo}[2]{
\begin{pmatrix} #1 \\ #2 \end{pmatrix}
}

\newsymbol\dnd 232D

\newcommand{\one}{{\rm 1\hspace*{-0.4ex} \rule{0.1ex}{1.52ex}\hspace*{0.2ex}}}

\renewcommand{\v}{\vec{v}}
\newcommand{\w}{\vec{w}}

\newcommand{\Fp}{\mathbb{F}_p}
\newcommand{\Fpx}{\mathbb{F}_p^\times}

\newcommand{\Z}{\mathbb Z}
\newcommand{\Q}{\mathbb Q}
\newcommand{\N}{\mathbb N}
\newcommand{\R}{\mathbb R}
\newcommand{\C}{\mathbb C}
\newcommand{\E}{\mathbb E}
\newcommand{\F}{\mathbb F}

\renewcommand{\a}{\alpha}
\renewcommand{\b}{\beta}
\renewcommand{\d}{\delta}
\newcommand{\e}{\epsilon}
\renewcommand{\t}{\theta}
\newcommand{\z}{\zeta}
\newcommand{\s}{\sigma}

\renewcommand{\l}{\ell}

\newcommand{\Lc}{\mathcal L}
\renewcommand{\Mc}{\mathcal M}
\newcommand{\Pc}{\mathcal P}
\newcommand{\Fc}{\mathcal F}
\newcommand{\Gc}{\mathcal G}
\newcommand{\Vc}{\mathcal V}

\renewcommand{\exp}[1]{e\!\left(#1\right)}

\newcommand{\ignore}[1]{}

\newcommand*\circled[1]{\tikz[baseline=(char.base)]{
            \node[shape=circle,draw,inner sep=2pt] (char) {#1};}}

\newcommand*\squared[1]{\tikz[baseline=(char.base)]{
            \node[shape=rectangle,draw,inner sep=2pt] (char) {#1};}}

\title{A converse to a theorem of Gauss on Gauss sums}

\author{Jonathan W.\ Bober}
\address{\hspace{-0.3in} Heilbronn Institute for Mathematical Research, School of Mathematics \\
University of Bristol \\
Bristol, UK}
\email{j.bober@bristol.ac.uk}

\author{Leo Goldmakher}
\address{\hspace{-0.3in} Dept of Mathematics and Statistics \\
Williams College \\
Williamstown, MA, USA  01267}
\email{Leo.Goldmakher@williams.edu}

\date{}


\begin{abstract}
Gauss proved that if $f$ is the quadratic character (mod $p$) then the magnitude of its Gauss sum is precisely $\sqrt{p}$. We prove a converse to this: if $f : \Fpx \to \C^\times$ has image $\{\pm 1\}$ and its Gauss sum has magnitude $\sqrt{p}$, we show that $f$ must be either the Legendre symbol (mod $p$) or its negation.
In fact, we prove a much more general result: under some mild hypotheses, $f$ is a nontrivial Dirichlet character (mod $p$) if and only if the Fourier transform of $f$ has magnitude $1$ somewhere in $\Fpx$.
This leads to a number of applications, with the common theme that extremal behavior on the Fourier side imposes multiplicative structure on the physical side.
\end{abstract}

\maketitle
\thispagestyle{empty}

\numberwithin{equation}{section}

\section{Introduction}

Consider an arbitrary function $f : \Fpx \to \C^\times$.
A famous result (essentially due to Gauss) is that the quantity
\[
\tau(f) :=
\sum_{n \in \F_p^\times}
f(n) \
\exp{\frac{n}{p}}
\]
has magnitude $\sqrt{p}$ whenever $f$ is a nontrivial character mod $p$ (here and throughout, $e(\a) := e^{2 \pi i \a}$). Does the converse hold? Our initial goal is to prove that it does, under some mild assumptions about $f$.
In the special case that $f$ produces a string of $\pm 1$'s, our approach gives a particularly clean result:
\begin{prop}\label{prop:QuadraticMainThm}
Suppose the image of $f : \Fpx \to \C^\times$ is $\{\pm 1\}$, and that $f(1) = 1$.
Then $f$ is a character (namely, the Legendre symbol) if and only if $|\tau(f)| = \sqrt{p}$.
\end{prop}
\noindent
Thus we can verify an algebraic condition (that $f$ is a homomorphism) via an analytic one (that $f$ correlates in a special way with exponentials).

There are two obvious paths to generalizing Proposition \ref{prop:QuadraticMainThm}. The first is to weaken the hypotheses, but we cannot be too cavalier---for $f$ to be a character it is necessary that $f(1) = 1$ and that the image of $f$ consists of roots of unity.
The second natural generalization is to interpret the quantity $\tau(f)$ in terms of the finite Fourier transform of $f$, and to attempt to replace our condition on $\tau(f)$ by a condition on $\wh{f}$.
More precisely, given any $f : \Fpx \to \C^\times$ we define its Fourier transform $\wh{f} : \F_p \to \C$ by
\[
\wh{f}(\xi) :=
\frac{1}{\sqrt{p}}
    \sum_{x \in \F_p} f(x) e\Big(-\frac{x\xi}{p}\Big) ,
\]
where $f(0):= 0$.
In this language, the conclusion of Proposition \ref{prop:QuadraticMainThm} asserts that $f$ is the Legendre symbol (mod $p$) if and only if $|\wh{f}(-1)| = 1$.
Does this hold if we replace $-1$ by some other input?
Our main theorem, which generalizes Proposition \ref{prop:QuadraticMainThm} in both the ways we just described, answers this question in the affirmative:

\begin{theorem}\label{thm:FourierCoeff}
Suppose the image of $f : \F_p^\times \to \C^\times$ consists of $n^\text{th}$ roots of unity, that $f(1) = 1$, and that $p \nmid n$.
Then $f$ is a nontrivial character if and only if $\exists a \in \Fpx$ such that $|\wh{f}(a)| = 1$.
\end{theorem}

\begin{rmk}
It was previously known that if $f : \Fpx \to \{\pm 1\}$, $f(0)=0$, and $|\wh{f}(a)| = 1$ for \emph{every} $a \in \Fpx$, then $f = \pm \leg{\cdot}{p}$; this has appeared explicitly as exercise \textbf{11} (iv) in Green's course on additive combinatorics \cite{Gree09}, but it also follows instantly from Theorem 1.1 of Borwein-Choi-Yazdani's work \cite{BCY} on Fekete polynomials, or by combining Lemma 1 in Kurlberg's paper \cite{Kurl02} with the main theorem of Bir\'o \cite{Biro} or Kurlberg \cite{Kurl02} on autocorrelations.
The primary interest of our result is that it shows that sampling $\wh f$ at a \emph{single} point is sufficient to determine whether or not $f$ is a character.
\end{rmk}

\begin{rmk}
The work of Borwein-Choi-Yazdani \cite{BCY} mentioned above is morally the closest to ours. Rewritten in our language, it asserts that if $N$ is odd and $f : \{1,2,\ldots,N-1\} \to \{\pm 1\}$ satisfies $\max\limits_{a \mod N} |\wh f(a)| = 1$, then $N$ must be prime and $\pm f$ must be the Legendre symbol mod $N$.
This illustrates a theme running through our present work: extremal behavior of the Fourier transform induces multiplicative structure on the physical side.
However, their result still requires sampling $\wh f$ at all points of $\Z/N\Z$. For example, for $N=15$ consider
\[
f(n) :=
\begin{cases}
1 & \mbox{if } 3 \mid n \text{ and } 5 \nmid n \\
-1 & \mbox{if } 3 \nmid n \text{ and } 5 \mid n \\
\leg{n}{15} & \mbox{otherwise},
\end{cases}
\]
Then $f$ isn't a character, even though $|\tau(f)| = \sqrt{15}$ (and more generally,
$|\wh f(a)| = 1$ for all $a \in (\Z/15\Z)^\times$).
\end{rmk}

\begin{rmk}
The hypothesis $p \nmid n$ in Theorem \ref{thm:FourierCoeff}
cannot be dropped, for a couple of reasons.
First, when $p \mid n$ it's possible for a non-character to have Gauss sum of magnitude $\sqrt{p}$; for example, taking $p = 3$ and setting $f(1) = 1$ and $f(2) = \exp{\frac{5}{6}}$ yields $|\tau(f)| = \sqrt{3}$.
(There are similar constructions for every odd prime $p$; see Proposition \ref{prop:FakeChar} below.)
Second, if $n$ is chosen to be minimal in Theorem \ref{thm:FourierCoeff}, then $p \mid n$ implies that $f$ is not a character.
To see this, observe that if $f$ were a character then $n$ would be the order of $f$ as an element of the dual group $\wh{\Fpx}$. This would imply $n \mid p-1$, whence $p \nmid n$.
\end{rmk}

\begin{rmk}
Recently, Benoist \cite{Beno24} used tools from symplectic geometry to show that the hypothesis that the image of $f$ consists of roots of unity in Theorem \ref{thm:FourierCoeff} cannot be dropped: for every $p \geq 11$ he constructs non-characters $f : \Fp \to \C$ such that $f(1)=1$, $f(0)=0 = \wh{f}(0)$, and $|f(a)| = 1 = |\wh{f}(a)|$ for all $a \neq 0$.
\end{rmk}

\subsection{Consequences}

Theorem \ref{thm:FourierCoeff} (and the ideas contained in our proof of it) implies a few results we haven't found in the literature.
Our first application of Theorem \ref{thm:FourierCoeff} is to the behavior of characters sampled along a function.
Given any nontrivial character $\chi \mod p$ and a function $P : \Fpx \to \Fpx$, consider
\[
\wh{\chi \circ P}(a) := \frac{1}{\sqrt{p}}
\sum_{x \in \Fpx} \chi\big(P(x)\big) e\Big(\!\!-\frac{ax}{p}\Big) .
\]
Whenever $P$ is injective, it's an exercise to prove that
\[
\max_{a \in \Fpx} |\wh{\chi \circ P}(a)| \geq 1 .
\]
We know that equality is achieved when $P$ is the identity map. For which other choices of $P$ does equality hold?
We shall prove that if $\chi$ is primitive (i.e.\ has order $p-1$), then equality is attained above if and only if $P(x) = cx^k$ for some $c \not\equiv 0 \mod p$ and some $k$ coprime to $p-1$.
More generally, in section \ref{sect:CharSumAlongFunction} we'll show:

\begin{cor}\label{cor:CharSumAlongFunction}
Given $P : \Fpx \to \Fpx$ arbitrary, $Q$ a permutation of $\Fpx$, and $\chi \mod p$ a nontrivial Dirichlet character such that
\[
\bigg| \sum_{x \in \Fpx} \chi\big(P(x)\big) e\Big(\frac{Q(x)}{p}\Big) \bigg| = \sqrt{p} .
\]
Then there exists $c \in \Fpx$ and $k \in \{1,2,\ldots,p-2\}$ such that
\(\displaystyle
\chi\Big( \frac{P(x)}{c \, Q(x)^k} \Big) = 1 .
\)
\end{cor}

\noindent
In the case that $P$ and $Q$ are polynomials, sums of the above form are classical. For example,
whenever $\chi$ is a nontrivial character (mod $p$),
$P \in \F_p[x]$ is nonconstant and squarefree,
and $Q \in \F_p[x]$ is nonconstant with $p \nmid \deg Q$, the Riemann Hypothesis for curves over finite fields implies
\begin{equation}\label{eq:KatzWeilBound}
\Big|
\sum_{x \in \F_p} \chi\big(P(x)\big) e\Big(\frac{Q(x)}{p}\Big)
\Big|
\leq (\deg P + \deg Q - 1) \sqrt{p} .
\end{equation}
(This appears as the final displayed equation in Weil's paper \cite{Weil48}.)
Our circle of ideas allows us to characterize when equality holds in Weil's bound (see section \ref{sect:ExtremeWeilBd}), a result which is surely well-known to experts but which we have not been able to find in the literature:

\begin{prop}\label{prop:ExtremeWeilBd}
Suppose equality holds in the Weil bound \eqref{eq:KatzWeilBound}. Then both $P$ and $Q$ are linear, and
\[
\sum_{x \in \F_p} \chi\big(P(x)\big) e\Big(\frac{Q(x)}{p}\Big)
= \chi\Big(\frac{P'(0)}{Q'(0)}\Big) e\Big(\frac{Q(r)}{p}\Big) \tau(\chi)
\]
where $r$ is the root of $P$.
\end{prop}

A different direction for our ideas is best phrased in the language of additive combinatorics, where finite Fourier transforms play an important role in indicating hidden structure inside of a given set.
Before describing our application, we set some notation.
Given a set $A \subseteq \Fpx$, we abuse notation and use $A$ to also denote its indicator function, i.e.\
\[
A(x) :=
\begin{cases}
1 & \mbox{if } x \in A \\
0 & \mbox{otherwise}.
\end{cases}
\]
Given $f : X \to \C$ we'll denote the expectation of $f$ by
\[
\E_X[f] := \frac{1}{|X|} \sum_{x \in X} f(x) ,
\]
so that for example the density of $A$ is $\d := \E_{\Fpx}[A]$.
Viewing $A$ as a subset of $\Fp$ implies $A(0)=0$, whence Fourier inversion tells us the mean of $\wh A$ over $\Fpx$:
\[
\E_{\Fpx} [ \wh A] =
-\frac{\d}{\sqrt{p}} .
\]
Parseval gives us the variance:
\[
\textup{Var}[\wh A]
= \E_{\Fpx}[ |\wh A|^2] - \frac{\d^2}{p}
= \d (1-\d) .
\]
To study extreme values of $\wh A$ we wish to maximize the variance, which happens at $\d = \frac{1}{2}$; for sets $A$ with this density we have $\E[\wh A] = -\frac{1}{2 \sqrt{p}}$ and $\textup{Var}[ \wh A] = \frac{1}{4}$.
This suggests renormalizing our random variable to $\frac{1}{\sqrt{p}} + 2 \wh A(a)$,
which (for $A$ of density $\frac{1}{2}$) has mean 0 and variance 1.
Thus whenever $|A| = \frac{p-1}{2}$ we have
\[
\max_{a \in \Fpx} \left| \frac{1}{\sqrt{p}} + 2 \wh{A}(a) \right|
\geq 1 .
\]
In section \ref{sect:CharacterizeQuadResid} we'll show that the choice of $A$ that minimizes the above has strong multiplicative structure:

\begin{prop}\label{prop:CharacterizeQuadResid}
Given an arbitrary set $A \subseteq \Fpx$. If there exists $a \neq 0$ such that
\(\displaystyle
\Big| \frac{1}{\sqrt{p}} +2 \wh{A}(a) \Big|
= 1
\)
then $A$ is either the set of quadratic residues or the set of quadratic nonresidues.
\end{prop}

All our results have concerned complete sums. One amusing geometric consequence of our main results is that an incomplete (nontrivial) Gauss sum can never be orthogonal to its complement. To make this precise, we require a bit of notation.
We will measure orthogonality via the standard real inner product on $\C$:
\[
z \perp w
\qquad \iff \qquad
\langle z,w \rangle := \Re z \overline{w} = 0 .
\]
As before, we denote the indicator function of $A$ by $A(x)$. In section \ref{sect:IncompleteOrthogonality} we'll prove

\begin{prop}\label{prop:IncompleteOrthogonality}
Suppose $\chi\mod p$ is a nontrivial Dirichlet character and that $A$ is a nonempty proper subset of $\Fpx$.
Then $\wh{\chi A}(a) \perp \wh{\chi A^c}(a)$ for some $a \in \Fpx$ if and only if $\chi$ is not quadratic and one of $A$ or $A^c$ is the set of all quadratic residues.
\end{prop}

\begin{cor}
For any $H \in [1,p-2]$ and any nontrivial Dirichlet character $\chi \mod p$, the incomplete Gauss sums
\[
\sum_{n \leq H} \chi(n) e\Big(\frac{an}{p}\Big)
\qquad \text{and} \qquad
\sum_{H < n \leq p-1} \chi(n) e\Big(\frac{an}{p}\Big)
\]
are not orthogonal for any $a \in \Fpx$.
\end{cor}

\begin{rmk}
In the case that $a=0$ it \emph{is} possible for the sum and its complement to be orthogonal, but only in the degenerate case where both are actually 0.
\end{rmk}

\begin{proof}[Proof of Corollary]
Let $A$ be the set of positive integers $\leq H$, and suppose for some nonzero $a$ the two incomplete Gauss sums were orthogonal. By Proposition \ref{prop:IncompleteOrthogonality}, $A$ must be the set of all quadratic residues (since it contains $1$). But this would imply $H = \frac{p-1}{2}$, which is impossible: if $p = 3$ the only nontrivial character is quadratic, and if $p \geq 5$ then $2$ and $\overline{2} = \frac{p+1}{2}$ are either both quadratic residues or both quadratic nonresidues.
\end{proof}

Thus far, all of our results start with an exact evaluation on the Fourier side and deduce multiplicative structure on the physical side.
It turns out we can deduce this type of conclusion even if we relax the hypothesis.
For ease of reference, we isolate the main technical hypothesis from Theorem \ref{thm:FourierCoeff}: given $f : \Fpx \to \C^\times$ we will typically assume that
\begin{equation}\label{eq:Hypotheses}
f(\Fpx) \subseteq \mu_n
\qquad \text{and} \qquad
p \nmid n ,
\end{equation}
where $\mu_n$ denotes the set of all $n^\text{th}$ roots of unity.
In section \ref{sect:FourierCoeff} we'll use some basic algebraic number theory to rewrite Theorem \ref{thm:FourierCoeff} in a more surprising form:

\begin{cor}\label{cor:FourierCoeff}
Suppose $f : \F_p^\times \to \C^\times$ satisfies \eqref{eq:Hypotheses} and $f(1) = 1$.
Then $f$ is a nontrivial character if and only if $\exists a \in \Fpx$ such that $|\wh{f}(a)| \in \Q$.
\end{cor}

\begin{rmk}
Our method of proof shows one can replace $\Q$ by a larger field, e.g.\ $\Q(\z_m)$ for any $m\not\equiv 0 \mod p$.
\end{rmk}
\noindent
It's a classical fact that $|\wh{f}(a)|=1$ for all nonzero $a$ whenever $f$ is a character (mod $p$). Thus our results imply the following curious dichotomy that doesn't explicitly mention characters:

\begin{cor}
Suppose $f : \F_p^\times \to \C^\times$ satisfies \eqref{eq:Hypotheses}.
Then
\[
|\wh{f}(a)| \notin \Q \quad \forall a \in \Fpx
\qquad \text{or} \qquad
|\wh{f}(a)| = 1 \quad \forall a \in \Fpx .
\]
\end{cor}

Our final application demonstrates that our ideas are relevant in situations where we only have access to the average behavior of the Fourier transform, rather than to exact magnitudes at any particular frequency.
Recall that the squared $L^2$ norm of $\wh{f}$ is $p-1$, so we expect that $|\wh{f}(a)| \approx 1$ on average over $\Fpx$.
In section \ref{sect:BlurredFourierCharacter} we'll show that a precise version of this heuristic imposes strong multiplicative structure on $f$:

\begin{prop}\label{prop:BlurredFourierCharacter}
Given $f : \Fp \to \C$ satisfying $f(0) = 0$ and the usual hypotheses \eqref{eq:Hypotheses}.
Pick any $K : \Fp \to \C$ such that $\wh{K}(a) \neq 0$ for all $a \neq 0$.
If
\[
\sum_{a \in \Fp} K(a+b) |\wh{f}(a)|^2
= \sum_{a \in \Fpx} K(a+b)
\]
for all $b \in \Fp$ then $f = \epsilon \chi$ for some $\epsilon \in \mu_n$ and some nontrivial character $\chi \mod p$.
\end{prop}

\begin{rmk}
Note that the sum on the left hand side runs over all of $\Fp$, while the sum on the right only runs over $\Fpx$. Conceptually this is because $f$ is properly viewed as a function on $\Fpx$, while $\wh f$ is a function on $\Fp$.
\end{rmk}

\begin{rmk}
At first glance, it may appear that the hypothesis---knowledge about the typical behavior of $|\wh f|$---is significantly weaker than our previous hypotheses. The key difference is that we must sample these averages at a large number of points (all $b \in \Fp$), rather than at a single point. Put differently, we trade a single perfectly sharp image for a large number of blurred images.
\end{rmk}

\noindent
Our proof of Proposition \ref{prop:BlurredFourierCharacter} is rather robust, and readily admits modifications. For example, we may replace the condition that $\wh K(a) \neq 0$ for all $a \neq 0$ by the weaker condition that at least one of $\wh K(a),\wh K(-a)$ is nonzero for all $a \neq 0$, and still draw the same conclusion.

\subsection{Other characterizations of characters}

We conclude our introduction with a brief survey of some other characterizations of characters.
Recall one of the standard definitions: a function $f : \Z \to \C$ (or $\N \to \C$) is a Dirichlet character (mod $q$) if and only if
\begin{itemize}
\item $f$ is completely multiplicative,
\item $f$ is periodic with period $q$, and
\item the support of $f$ is the set of integers coprime to $q$.
\end{itemize}
It turns out there are various ways to modify these hypotheses without including any other functions.
One example of this is the following result due to Allouche and the second author \cite{AlloGold18}, which built on previous work of S\'{a}rk\"{o}zy \cite{Sark78}, Heppner-Maxsein \cite{HeppMaxs85}, and Methfessel \cite{Meth94}:

\begin{prop}[Allouche-Goldmakher]
A function $f : \Z \to \C$ is a Dirichlet character if and only if $f$ is completely multiplicative, eventually satisfies a linear recurrence, and has support strictly larger than $\{\pm 1\}$.
\end{prop}
\noindent
There are also a number of related characterizations that use the notion of automaticity in place of linear recurrences; see \cite{KlurKurl20} and \cite{Koni20} for a representative example.

A different genre of characterization, discovered recently by Konieczny \cite{Koni24}, starts with the observation that for any Dirichlet character $\chi$ and any integer $a \geq 0$, the function $\chi(n) n^a$ agrees with the outputs of a \emph{generalized polynomial}, i.e.\
any function that can be built out of polynomials, addition, multiplication, and the floor function.

\begin{theorem}[Konieczny]
Suppose $f : \N \to \C$ is a completely multiplicative function that coincides with a generalized polynomial, and that $f$ has support strictly larger than $\{1\}$. Then there exists a Dirichlet character $\chi$ and an integer $a \geq 0$ such that $f(n) = \chi(n) n^a$.
\end{theorem}

A third type of characterization is in terms of the asymptotic behavior of the mean value.
Orthogonality implies that for any Dirichlet character $f$,
\begin{equation}\label{eq:MeanValueKM}
\sum_{n \leq x} f(n) = \a x + O(1)
\end{equation}
for all large $x$ (note that $\a=0$ if and only if $f$ is nontrivial).
Chudakov conjectured the converse should hold, and this was proved by Glazkov \cite{Glaz68} in the case that $\a \neq 0$ and by Klurman-Mangerel \cite{KlurMang18} in the case that $\a = 0$:

\begin{theorem}[Glazkov, Klurman-Mangerel]
Suppose $f : \N \to \C$ is a completely multiplicative function whose image is finite and whose support contains all but finitely many primes.
If \eqref{eq:MeanValueKM} holds for some $\a \in \C$,
then $f$ must be a Dirichlet character.
\end{theorem}

All three of the above results, as well as almost all other characterizations of characters we've seen, assume from the outset that $f$ is multiplicative.
Apart from our main theorem (Theorem \ref{thm:FourierCoeff}), the only\footnote{A very recent preprint by Hiary and Saraeb \cite{Hiary} deduces multiplicativity from a functional equation.} other characterization we're aware of that \emph{deduces} multiplicativity from other hypotheses is the following result due to Kurlberg \cite{Kurl02}:

\begin{theorem}[Kurlberg] \label{thm:Kurlberg}
Suppose $f : \F_p \to \C$ satisfies $f(0) = 0$ and $f(1) = 1$, and that the image of $\Fpx$ under $f$ consists of roots of unity. Then $f$ is a nontrivial multiplicative character on $\F_p$ if and only if
\[
\sum_{x \in \F_p} f(x) \overline{f(x+h)} =
\begin{cases}
-1 & \mbox{if } h \neq 0 \\
p-1 & \mbox{otherwise.}
\end{cases}
\]
\end{theorem}
\noindent
In other words, the behavior of a character (mod $p$) is completely determined by its autocorrelations. This result was motivated by a question of Harvey Cohn from the 1990s, who asked whether this happens more generally.
It turns out that the natural analogue of Kurlberg's result in other finite fields fails to hold, as demonstrated by Choi and Siu; for any odd $p$ and any $k \geq 2$, Theorem 3 of \cite{ChoiSiu00} produces a non-character $f : \F_{p^k}^\times \to \{\pm 1\}$ that satisfies the analogous autocorrelation condition.

In the next section, we shall make crucial use of Kurlberg's theorem \ref{thm:Kurlberg} (with the additional restriction that $p \nmid n$) to prove Theorem \ref{thm:FourierCoeff}.
Indeed, Theorem \ref{thm:FourierCoeff} is equivalent to this restricted version of Kurlberg's result.
To see this, suppose $f : \F_p \to \C$ satisfies $f(0) = 0$ and $f(1) = 1$, and that its image consists of $n^\text{th}$ roots of unity with $p \nmid n$. If
\[
\sum_{x \in \F_p} f(x) \overline{f(x+h)} =
\begin{cases}
-1 & \mbox{if } h \neq 0 \\
p-1 & \mbox{otherwise,}
\end{cases}
\]
then expanding $|\tau(f)|^2$ as a double sum, re-indexing, and plugging in the above autocorrelation identity yields $|\tau(f)|^2 = p$, whence (by Theorem \ref{thm:FourierCoeff}) $f$ must be a nontrivial character.

\section{Proof of Theorem \ref{thm:FourierCoeff}}

Abusing nomenclature, we call the quantity $\tau(f)$ the \emph{Gauss sum} of $f$.
Our goal is to show that the value of the Gauss sum of a function carries a remarkable amount of information about the behavior of the function itself.
As a warm up, observe that if $g(x) = -1$ for all $x \in \Fpx$, then $\tau(g) = 1$; a bit of playing around shows that this constant function is the \emph{only} function $\Fpx \to \{\pm 1\}$ that has Gauss sum $1$.
This phenomenon generalizes to functions whose Gauss sums live in any finite abelian extension of $\Q$:

\begin{lemma}\label{lem:RatlGaussSum}
Suppose $g : \F_p^\times \to \Q(\zeta_n)$ such that $p \nmid n$ and $\tau(g) \in \Q(\zeta_n)$.
Then $g(a) = -\tau(g)$ for all $a \in \Fpx$.
\end{lemma}

\begin{proof}
Note that the polynomial
\[
m(x) := -\tau(g) + g(1) x + g(2) x^2 + \cdots + g(p-1)x^{p-1}
\]
has $\zeta_p$ as a root.
This means $m(x)$ is a multiple of the minimal polynomial of $\zeta_p$ over $\Q(\zeta_n)$, which (since $p \nmid n$) is $1+x+\cdots+x^{p-1}$.
The claim instantly follows.
\end{proof}

\noindent
Our proof of Theorem \ref{thm:FourierCoeff} proceeds in several stages of increasing generality.
We start by proving a converse to Gauss' theorem on Gauss sums.

\begin{prop}\label{prop:ConverseToGauss}
Suppose the image of $f : \F_p^\times \to \C^\times$ consists of $n^\text{th}$ roots of unity, that $f(1) = 1$, and that $p \nmid n$.
Then $f$ is a nontrivial character if and only if $|\tau(f)| = \sqrt{p}$.
\end{prop}

\begin{proof}
The forward direction is classical, so we henceforth assume $|\tau(f)| = \sqrt{p}$ and prove that $f$ is a character.
Extend $f$ to a function $\F_p \to \C$ by setting $f(0) := 0$, and observe that
\[
p
= \left| \sum_{\l \in \F_p} f(\l) \ \exp{\frac{\l}{p}} \right|^2
= \sum_{m,\l \in \F_p} f(m) \overline{f(\l)} \ \exp{\frac{m-\l}{p}}
= \sum_{k \in \F_p} \exp{\frac{k}{p}} \sum_{\l \in \F_p} f(\l+k) \overline{f(\l)} .
\]
The contribution from the $k=0$ term is $\sum\limits_{a \in \F_p} |f(a)|^2 = p-1$, whence
\[
\sum\limits_{k \in \Fpx}
\Bigg( \underbrace{\sum_{\l \in \F_p} f(\l+k) \overline{f(\l)}}_{g(k)} \Bigg) \exp{\frac{k}{p}} = 1 .
\]
Note that $g(k) \in \Q(\zeta_n)$ and $\tau(g) = 1$, so Lemma \ref{lem:RatlGaussSum} implies $g(k) = -1$ for all $k \in \Fpx$. Since $g(0) = p-1$, Kurlberg's theorem \ref{thm:Kurlberg} implies that $f$ must be a character.
\end{proof}

\noindent
In practice, it's helpful to relax the hypotheses by removing the condition that $f(1) = 1$. The result becomes simpler to state if we abuse notation and refer to a \emph{function} $h : \Fpx \to \C^\times$ as an $n^\text{th}$ root of unity if and only if everything in the image of $h$ is an $n^\text{th}$ root of unity.

\begin{cor}\label{cor:ConverseToGauss}
Suppose $f : \F_p^\times \to \C^\times$ is an $n^\text{th}$ root of unity, and that $p \nmid n$. Then $|\tau(f)| = \sqrt{p}$ if and only if $f = \epsilon \chi$ for some constant $\epsilon$ and some nontrivial character $\chi$ that are both $n^\text{th}$ roots of unity.
\end{cor}

\begin{proof}
The reverse direction immediately follows from Gauss' theorem, so we assume $|\tau(f)| = \sqrt{p}$.
The function $g := \overline{f(1)} f$ satisfies all the hypotheses of Proposition \ref{prop:ConverseToGauss}, so $g$ must be a nontrivial character.
\end{proof}

\noindent
Armed with these tools, we can now give a short proof of our main theorem.

\begin{proof}[Proof of Theorem \ref{thm:FourierCoeff}]
Recall our claim:
given $f : \F_p^\times \to \C^\times$ an $n^\text{th}$ root of unity such that $f(1) = 1$ and $p \nmid n$, we wish to prove
that $f$ is a nontrivial character if and only if $\exists a \in \Fpx$ such that $|\wh{f}(a)| = 1$.
As before, the forward direction is classical (take $a = -1$), so it suffices to prove the reverse direction. Pick $a \in \Fpx$ such that $|\wh{f}(a)| = 1$.
A change of variables yields
\[
\wh{f}(a)
= \frac{1}{\sqrt{p}}
    \sum_{m \in \Fpx} f(-\overline{a} m) \ e\Big(\frac{m}{p}\Big) ,
\]
where $\overline{a}$ denotes the multiplicative inverse of $a$ in $\Fpx$.
Corollary \ref{cor:ConverseToGauss} implies
\(
f(-\overline{a} m) = \epsilon \chi(m)
\)
for some $n^\text{th}$ roots of unity $\epsilon$ and $\chi$ (with the former a constant and the latter a nontrivial character).
This in turn implies
\[
f(\l) = \epsilon \chi(-a) \chi(\l) = \epsilon' \chi(\l)
\]
with $\epsilon'$ an $n^\text{th}$ root of unity.
Since $f(1) = 1 = \chi(1)$, we deduce $\epsilon' = 1$, and the claim is proved.
\end{proof}

\section{False characters} \label{sect:FalseCharacters}

In the introduction we observed that the function $f : \F_3 \to \C$ defined $f(0)=0$, $f(1) = 1$, and $f(2) = \exp{\frac{5}{6}}$ satisfies $|\tau(f)| = \sqrt{3}$, despite not being a character.
This is not an isolated example:

\begin{prop}\label{prop:FakeChar}
For every prime $p \geq 5$ there exists a function $f : \Fp \to \C$ such that $f(0)=0$, $f(1)=1$, $f(\Fpx) \subseteq \mu_{2p}$, and $|\wh{f}(a)| = 1$ for some nonzero $a$, but $f$ is not a character.
\end{prop}
\noindent
In other words, the condition $p \nmid n$ in the statement of Theorem \ref{thm:FourierCoeff} cannot be removed.
The construction given below, proposed by ChatGPT, satisfies the additional character-like property $\wh f(0)=0$.

\begin{proof}
Let $\chi$ denote the quadratic character mod $p$, and pick $b$ such that $\chi(b) = -1$. Set
\[
f(n) := \chi(n) e\bigg( \frac{(n-1)+b(\overline{n} - 1)}{p} \bigg)
\]
where $\overline{n}$ denotes the inverse of $n$ in $\Fpx$ and $\overline{0} := 0$.
A straightforward computation shows that $\wh f(a)$ can be rewritten as a Sali\'e sum multiplied by a normalizing factor; evaluating the Sali\'e sum (see e.g.\ \cite{IK}) shows that for any $a \in \Fp$,
\[
\wh f(a) =
\e \sum_{u^2 = 4b(1-a)} e\Big(\frac{u}{p}\Big)
\]
where
\(
\e :=
-\frac{\tau(\chi)}{\sqrt{p}} e\big(\!\!-\!\!\frac{b+1}{p}\big)
\)
is a root of unity. In particular,
\(
\wh f(0) = 0
\)
and
\(
|\wh f(1)| = 1 .
\)
However, $f$ cannot be a character:
$e\big( \frac{(n-1)+b(\overline{n} - 1)}{p} \big)$
is a nontrivial $p$-th root of unity for any $n \notin \{0,1,b\}$ so $f(n)$ has order divisible by $p$, but any character has order dividing $p-1$.
\end{proof}

\section{Proof of Corollary \ref{cor:CharSumAlongFunction}}
\label{sect:CharSumAlongFunction}

Recall that Corollary \ref{cor:CharSumAlongFunction} asserts that
if
\begin{equation}\label{eq:ExtremeCharSumAlongFunct}
\bigg| \sum_{x \in \Fpx} \chi\big(P(x)\big) e\Big(\frac{Q(x)}{p}\Big) \bigg| = \sqrt{p}
\end{equation}
for some nontrivial $\chi \mod p$,
some function $P : \Fpx \to \Fpx$, and
some permutation $Q$ of $\Fpx$,
then there exists $c \in \Fpx$ and $k \in \{1,2,\ldots,p-2\}$ such that
\begin{equation}\label{eq:KernelOfChi}
\chi\Big( \frac{P(x)}{c \, Q(x)^k} \Big) = 1 .
\end{equation}

\begin{proof}
Since $Q$ is a permutation and there are no restrictions on $P$, a change of variables allows us to assume that $Q(x)=x$, so \eqref{eq:ExtremeCharSumAlongFunct} takes the form
\[
|\tau(\chi \circ P)| =
\bigg| \sum_{x \in \Fpx} \chi\big(P(x)\big) e\Big(\frac{x}{p}\Big) \bigg| = \sqrt{p} .
\]
Corollary \ref{cor:ConverseToGauss} implies that
\[
\chi \circ P = \e \xi
\]
for some $\e \in \mu_d$ and some nontrivial character $\xi \mod p$ of order dividing $d$, where $d$ denotes the order of $\chi$.
Since the group of characters $\wh \Fpx$ is cyclic, the subgroup consisting of all characters of order dividing $d$ must be as well. In particular, $\xi = \chi^k$ for some $k \in \{1, 2, \ldots,d-1\}$.
Also, since $\chi$ has order $d$ there exists $c \in \Fpx$ such that $\e = \chi(c)$.
It follows that $\forall x \in \Fpx$,
\[
\chi \big(P (x)\big) = \chi (c x^k)
\]
or equivalently
\(
\chi\big( \frac{P(x)}{c \, x^k} \big) = 1 .
\)
\end{proof}

\section{Proof of Proposition \ref{prop:ExtremeWeilBd}}
\label{sect:ExtremeWeilBd}

Recall that we're assuming that equality holds in the Weil bound \eqref{eq:KatzWeilBound}, i.e.\
\[
\Big|
\sum_{x \in \F_p} \chi\big(P(x)\big) e\Big(\frac{Q(x)}{p}\Big)
\Big|
= (\deg P + \deg Q - 1) \sqrt{p}
\]
and we're trying to deduce that $P$ and $Q$ are both linear (and then evaluate the sum on the left hand side).

\begin{proof}
Let $C := \deg P + \deg Q -1$, and set
\[
S(a) :=
\sum_{x \in \F_p} \chi\big(P(x)\big) e\Big(\frac{a Q(x)}{p}\Big) .
\]
Thus, for example,
\[
|S(1)| = C \sqrt{p} .
\]
Our first goal is to show that $|S(a)| = C \sqrt{p}$ for all $a \in \Fpx$.
This step of the proof mimics Kurlberg's proof of Theorem \ref{thm:Kurlberg}.

Let $n$ denote the order of $\chi$ and consider the cyclotomic extensions $K := \Q(\z_n)$ and $L := K(\z_p)$, where we use the familiar notation $\z_m := e(\sfrac{1}{m})$.
We have $\Gal{L/K} \simeq \Fpx$, with each $a \in \Fpx$ corresponding to the automorphism $\s \in \Gal{L/K}$ that fixes $\z_n$ and maps $\z_p \mapsto \z_p^a$.
It follows that for every $a \in \Fpx$, the corresponding automorphism $\s \in \Gal{L/K}$ maps $S(1) \mapsto S(a)$.
Since $|S(1)|^2 \in \Q$ and $\s$ commutes with complex conjugation (as $\Gal{L/K}$ is abelian), we deduce that $|S(a)|^2 = |S(1)|^2 = C^2 p$, whence
\[
\phantom{\forall a \in \Fpx . \qquad}
|S(a)| = C \sqrt{p} \qquad \forall a \in \Fpx .
\]
The rest of the proof is most easily expressed in terms of the Fourier transform. First observe that
\[
\frac{1}{\sqrt{p}} S(a)
= \wh{F}(a)
\]
where
\[
F(t) := \sum_{\substack{x \in \Fp \\ Q(x) = -t}} \chi\big(P(x)\big) .
\]
Parseval implies
\[
\begin{split}
\frac{1}{p} \sum_{a \in \Fp} |S(a)|^2
&= \sum_{a \in \Fp} |\wh F(a)|^2
= \sum_{t \in \Fp} |F(t)|^2 \\
&\leq \sum_{t \in \Fp} \big|\{x \in \Fp : Q(x) = -t, P(x) \neq 0\}\big|^2 \\
&\leq (\deg Q) \sum_{t \in \Fp} \big|\{x \in \Fp : Q(x) = -t, P(x) \neq 0\}\big| \\
&\leq (\deg Q)p .
\end{split}
\]
On the other hand,
\[
\frac{1}{p} \sum_{a \in \Fp} |S(a)|^2
\geq \frac{1}{p} \sum_{a \in \Fpx} |S(a)|^2
= (p-1)C^2 .
\]
We deduce
\[
(\deg Q)^2
\leq C^2
\leq \frac{(\deg Q) p}{p-1} ,
\]
whence $\deg Q = 1$. Thus $C = \deg P$, which implies
\[
(\deg P)^2 = C^2 \leq \frac{p}{p-1}
\]
and we conclude that $\deg P = 1$ as well.

Because $P$ is linear we can write $P(x)=\a (x-r)$ for some $\a \in \Fpx$ and $r \in \Fp$. Since $Q$ is linear we can express it in the form
\(
Q(x)=\b (x-r) + Q(r)
\)
for some $\b \in \Fpx$.
It follows that
\[
\begin{split}
\sum_{x \in \Fp} \chi\big(P(x)\big) e\Big(\frac{Q(x)}{p}\Big)
&=
\chi(\a) e\Big(\frac{Q(r)}{p} \Big)
\sum_{y \in \Fp} \chi(y) e\Big(\frac{\b y}{p} \Big)
=
\chi(\a) e\Big(\frac{Q(r)}{p}\Big) \overline{\chi(\b)} \tau(\chi) \\
&=
\chi(\a \overline{\b})
e\Big(\frac{Q(r)}{p}\Big) \tau(\chi)
\end{split}
\]
as claimed.
\end{proof}

\section{Proof of Proposition \ref{prop:CharacterizeQuadResid}}
\label{sect:CharacterizeQuadResid}

The claim is vacuous if $A$ is empty or all of $\Fpx$, so we henceforth assume that $A$ is a nonempty proper subset of $\Fpx$.
Suppose that there exists $a_0 \in \Fpx$ such that
\[
\left|
\frac1{\sqrt p}+2\widehat A(a_0)
\right|
= 1 .
\]
Set $A^c := \Fpx \setminus A$ and consider the function
\[
B(x) := A(x)-A^c(x) =
\begin{cases}
0 & \mbox{if } x = 0 \\
1 & \mbox{if } x \in A \\
-1 & \mbox{if } x \in A^c .
\end{cases}
\]
Observe that for any $b \in \Fpx$ we have
\[
\wh B(b) =
\frac{1}{\sqrt{p}} \sum_{x \in \Fp} B(x) e(-bx/p)
= 2\frac{1}{\sqrt{p}} \sum_{x \in A} e(-bx/p) - \frac{1}{\sqrt{p}} \sum_{x \in \Fpx} e(-bx/p)
= 2 \wh A(b) + \frac{1}{\sqrt{p}} ,
\]
whence $|\wh B(a_0)| = 1$.
Since $A$ is proper and nonempty, the image of $B$ on $\Fpx$ is $\{\pm 1\}$, so corollary \ref{cor:ConverseToGauss} implies
\[
B = B(1) \chi
\]
for the quadratic character $\chi \mod p$.
If $B(1) = 1$ then $A$ is the set of quadratic residues; if $B(1) = -1$, $A$ is the set of quadratic nonresidues.
\hfill \qed

\section{Proof of Proposition \ref{prop:IncompleteOrthogonality}}
\label{sect:IncompleteOrthogonality}

Recall that we're given a nonempty proper subset $A \subset \Fpx$ and are trying to prove that
\begin{equation}\label{eq:OrthogHypothesis}
\wh{\chi A}(a) \perp \wh{\chi A^c}(a)
\end{equation}
for some $a \in \Fpx$
if and only if
$\chi$ is not quadratic and one of $A$ or $A^c$ is the set of all quadratic residues.

\begin{proof}
Observe that
\(
\wh{\chi A}(b) + \wh{\chi A^c}(b) = \wh \chi (b) ,
\)
whence
\[
|\wh{\chi A}(b) + \wh{\chi A^c}(b)| = 1
\]
for all $b \in \Fpx$.
On the other hand, we have
\[
\wh{\chi A} - \wh{\chi A^c}
= \wh{\chi B}
\]
where $B(x) := A(x) - A^c(x)$ as in the proof of Proposition \ref{prop:CharacterizeQuadResid}.
For any $b \in \Fpx$ we have
\[
4 \Big\langle \wh{\chi A}(b) , \wh{\chi A^c}(b) \Big\rangle
= |\wh{\chi A}(b) + \wh{\chi A^c}(b)|^2 - |\wh{\chi A}(b) - \wh{\chi A^c}(b)|^2
= 1 - |\wh{\chi B}(b)|^2 .
\]
It follows that
\begin{equation}\label{eq:InnerProdEquiv}
\wh{\chi A}(b) \perp \wh{\chi A^c}(b)
\qquad \iff \qquad
|\wh{\chi B}(b)| = 1 .
\end{equation}
Armed with this equivalence, the proof of our claim is relatively quick.

First, suppose there exists $a \in \Fpx$ such that \eqref{eq:OrthogHypothesis} holds. \eqref{eq:InnerProdEquiv} implies
$|\wh{\chi B}(a)| = 1$, whence
\(
B(1) \chi B
\)
is a nontrivial character (mod $p$) by Theorem \ref{thm:FourierCoeff}.
This in turn implies $B(1) B$ is a character (mod $p$).
Since $B(1) B$ is real-valued, it must either be the trivial character or the quadratic character. The former is impossible (because $A$ is assumed to be a nonempty proper subset of $\Fpx$), so
\[
B(x) = B(1) \leg{x}{p}
\]
for all $x \in \Fp$.
If $B(1) = 1$ this implies $A$ is the set of quadratic residues; if $B(1) = -1$ this implies $A^c$ is the set of quadratic residues.
Finally, note that $\chi$ cannot be quadratic: if it were, then $B(1) \chi B$ would be a constant function, contradicting our earlier deduction that it's a nontrivial character.

Now we prove the converse. Suppose $\chi$ is not the quadratic character and that $A$ or $A^c$ is the set of quadratic residues.
Then $B = \pm \leg{\cdot}{p}$, whence
\[
\chi B = \pm \leg{\cdot}{p} \chi .
\]
Since $\chi$ isn't quadratic, $\leg{\cdot}{p} \chi$ is a nontrivial character, whence
$|\wh {\chi B}(b)| = 1$ for all $b \in \Fpx$.
Applying \eqref{eq:InnerProdEquiv} concludes the proof.
\end{proof}

\section{Proof of Corollary \ref{cor:FourierCoeff}}
\label{sect:FourierCoeff}

Since the forward direction of the claim is immediate, we focus on the reverse: we'll prove that if there exists $a \in \Fpx$ such that
\[
|\wh f(a)| \in \Q
\]
(where $f$ satisfies the usual hypothesis \eqref{eq:Hypotheses}) then $f = \e \chi$ for some $\e \in \mu_n$ and some nontrivial $\chi \mod p$.

\begin{proof}
Fix $a \in \Fpx$ such that
\[
q := |\wh f(a)| \in \Q .
\]
As in the proof of Proposition \ref{prop:ExtremeWeilBd}, the key idea is to use basic algebraic number theory to deduce that $|\wh f|$ is constant on $\Fpx$.
Set
\[
\tau_a :=
\sum_{x \in \Fpx} f(x) e\Big(-\frac{ax}{p}\Big) .
\]
We can rewrite our hypothesis in the form
\(
|\tau_a| = q \sqrt{p} ,
\)
and it follows that
\begin{equation}\label{eq:NormTwistedGaussSum}
\tau_a \overline{\tau_a} = p q^2 .
\end{equation}
We immediately gain some nontrivial information from this. Since $\tau_a$ is an algebraic integer, the left hand side of \eqref{eq:NormTwistedGaussSum} must be as well, while the right hand side is rational, whence
\(
p q^2 \in \Z .
\)
This in turn implies $q \in \Z$. By Corollary \ref{cor:ConverseToGauss} it suffices to show that $q=1$.

Recall that each automorphism in $\Gal{L/K}$ (where $K = \Q(\z_n)$ and $L = K(\z_p)$) acts on $L$ by fixing $\z_n$ and mapping $\z_p \mapsto \z_p^b$ for some $b \in \Fpx$, and
that $\Gal{L/K} \simeq \Fpx$. Thus for each $b \in \Fpx$ there exists $\s \in \Gal{L/K}$ such that
\[
\s(\tau_a) = \tau_{ab} .
\]
Also, $\s$ commutes with complex conjugation, since $\Gal{L/\Q}$ is abelian. Applying $\s$ to both sides of \eqref{eq:NormTwistedGaussSum} yields
\(
\tau_{ab} \overline{\tau_{ab}} = p q^2
\)
for all $b \in \Fpx$, whence
\(
|\tau_c|^2 = p q^2
\)
for all $c \in \Fpx$.
It follows that
\[
|\wh f(c)| = q
\]
for all $c \in \Fpx$.
Parseval implies
\[
(p-1) q^2 \leq
\sum_{c \in \Fp} |\wh f(c)|^2 =
p-1 ,
\]
whence $q = 0$ or $1$. All that remains is to prove that $q$ is nonzero.

Suppose $q$ were zero. Then
\[
\tau(f) = \tau_{-1} = 0 ,
\]
and Lemma \ref{lem:RatlGaussSum} would imply $f(x) = -\tau(f) = 0$ for all $x \in \Fpx$, contradicting our hypotheses on $f$. Thus $q \neq 0$, and the claim is proved.
\end{proof}

\section{Proof of Proposition \ref{prop:BlurredFourierCharacter}}
\label{sect:BlurredFourierCharacter}

Recall that we're given some $f : \Fp \to \C$ satisfying $f(0) = 0$ and the usual hypotheses \eqref{eq:Hypotheses}, and some $K : \Fp \to \C$ satisfying $\wh{K}(a) \neq 0$ for all $a \neq 0$.
We're trying to prove that if
\begin{equation}\label{eq:Blur}
\sum_{a \in \Fp} K(a+b) |\wh{f}(a)|^2
= \sum_{a \in \Fpx} K(a+b)
\end{equation}
for all $b \in \Fp$ then $f = \epsilon \chi$ for some $\epsilon \in \mu_n$ and some nontrivial character $\chi \mod p$.

\begin{proof}
Let
\[
D(a) := |\wh f(a)|^2 - \chi_1(a)
\]
where $\chi_1$ denotes the trivial character (mod $p$).
Our hypothesis \eqref{eq:Blur} is equivalent to
\[
\sum_{a \in \F_p} K(a+b) D(a) = 0
\]
for every $b \in \Fp$, which we can rewrite in terms of additive convolutions:
\[
\whh{K} * D \equiv 0 .
\]
Applying the Fourier transform to both sides yields
\(
\wh K(-r) \wh D(r) = 0
\)
for all $r \in \Fp$. By our hypothesis on the nonvanishing of $\wh K$, we deduce
\[
\wh D(r) = 0
\]
for all nonzero $r$. We claim this also holds for $r=0$: Parseval implies
\[
\sum_{a \in \Fp} D(a) = \sum_{a \in \Fp} |\wh f(a)|^2 - (p-1) = 0 ,
\]
whence $\wh D(0)=0$. Thus $\wh D$ vanishes uniformly, whence $D \equiv 0$ as well. This implies $|\wh f(a)| = 1$ for every $a \in \Fpx$; applying Corollary \ref{cor:ConverseToGauss} yields the claim.
\end{proof}

\bigskip

\noindent
\textsc{Acknowledgements.}
We're grateful to Andrew Granville, Oleksiy Klurman, and Sacha Mangerel for their comments and encouragement, Ben Green for pointing out a relevant reference, Zoe Kane for some inspiring discussions, and ChatGPT for shooting down a number of our conjectures by constructing instructive counterexamples, most notably the one presented in Section \ref{sect:FalseCharacters}.
LG is grateful to the Centre de Recherche Math\'ematiques and the Simons Foundation for their support.

\end{document}